\documentclass[11pt,reqno,final]{amsart}

\usepackage[a4paper,left=35mm,right=35mm,top=30mm,bottom=30mm,marginpar=25mm]{geometry} 
\usepackage{amsmath,amsfonts,amssymb,amsthm,mathrsfs,amsopn}
\usepackage{graphicx,color}
\usepackage{latexsym}
\usepackage{cite}
\usepackage{hyperref}

\theoremstyle{plain}

\theoremstyle{plain}
\newtheorem{theorem}{Theorem}[section]
\newtheorem{corollary}[theorem]{Corollary}
\newtheorem{lemma}[theorem]{Lemma}

\theoremstyle{definition}

\theoremstyle{remark}

\numberwithin{equation}{section}

\newtheorem{remark}[theorem]{Remark}
\numberwithin{theorem}{section}
\numberwithin{equation}{section}
\numberwithin{figure}{section}

\def\D{\mathcal D}
\def\AA{\mathbb A}

\def\A{\mathscr A}
\def\M{\mathcal M}
\def\S{\mathcal S}

\def\L{\mathscr L}
\def\H{\mathscr H}
\def\H{\mathcal H}
\def\F{\mathscr F}

\def\d{{\rm  d}}

\def\1{\mathbf 1}

\newcommand{\Crm}{\mathrm{C}}

\newcommand{\Lrm}{\mathrm{L}}

\newcommand{\Urm}{\mathrm{U}}

\newcommand{\Dcal}{\mathcal{D}}
\newcommand{\Ecal}{\mathcal{E}}

\DeclareMathOperator*{\wslim}{w*-lim}

\DeclareMathOperator{\diverg}{div}
\DeclareMathOperator{\curl}{curl}

\newcommand{\ee}{\mathrm{e}}
\newcommand{\ii}{\mathrm{i}}

\newcommand{\setb}[2]{\bigl\{\, #1 \ \ \textup{\textbf{:}}\ \ #2 \,\bigr\}}

\newcommand{\setBB}[2]{\biggl\{\, #1 \ \ \textup{\textbf{:}}\ \ #2 \,\biggr\}}

\newcommand{\abs}[1]{|#1|}

\newcommand{\dprb}[1]{\bigl\langle #1 \bigr\rangle}

\newcommand{\di}{\mathrm{d}}
\newcommand{\dd}{\;\mathrm{d}}

\newcommand{\N}{\mathbb{N}}
\newcommand{\R}{\mathbb{R}}

\newcommand{\loc}{\mathrm{loc}}

\newcommand{\eps}{\epsilon}

\newcommand{\BV}{\mathrm{BV}}

\newcommand{\BD}{\mathrm{BD}}

\renewcommand{\epsilon}{\varepsilon}
\renewcommand{\hat}{\widehat}

\newcommand{\weakto} {\rightharpoonup}                 
\newcommand{\weakstarto}{\stackrel {*} {\weakto}}      
\newcommand{\aac}{\ll}
\newcommand{\defeq} {:=}

\newcommand{\mean}[1]{\,-\hskip-1.08em\int_{#1}} 

\newcommand{\res}{\mathop{\hbox{\vrule height 7pt width .5pt depth 0pt
\vrule height .5pt width 6pt depth 0pt}}\nolimits} 

\DeclareMathOperator{\imult}{\res} 

\DeclareMathOperator{\Ker}{ker}
\DeclareMathOperator{\Div}{div}
\DeclareMathOperator{\Span}{span}

\title[On the structure of ${\mathscr A}$-free measures]{On the structure of ${\mathscr A}$-free measures and applications}

\author[G.~De Philippis]{Guido De Philippis}
\address{\textit{Guido De Philippis:} Unit\'e de Math\'ematiques Pures et Appliqu\'ees -- ENS de Lyon, UMR 5669 -- UMPA, France.}
\email{guido.de-philippis@ens-lyon.fr}

\author[F.~Rindler]{Filip Rindler}
\address{\textit{Filip Rindler:} Mathematics Institute, University of Warwick, Coventry CV4 7AL, United Kingdom.}
\email{F.Rindler@warwick.ac.uk}


\begin{document}

\begin{abstract}
We establish a general structure theorem for the singular part of ${\mathscr A}$-free Radon measures, where ${\mathscr A}$ is a linear PDE operator. By applying the theorem to suitably chosen differential operators ${\mathscr A}$, we obtain a simple proof of Alberti's rank-one theorem and, for the first time,  its extensions to functions of bounded deformation (BD). We also prove a structure theorem for the singular part of a finite family of normal currents.  The latter result implies that the Rademacher theorem on the differentiability of Lipschitz functions can hold only for absolutely continuous measures and that every top-dimensional Ambrosio--Kirchheim metric current in $\mathbb R^d$ is a Federer--Fleming flat chain.

\vspace{4pt}


\vspace{4pt}

\vspace{4pt}

\noindent\textsc{Date:} \today{}.
\end{abstract}

\maketitle

\thispagestyle{empty}

\section{Introduction}

Consider a finite Radon measure $\mu$ on an open set $\Omega \subset \R^d$ with values in $\R^m$ that is $\A$-free for a $k$'th-order linear constant-coefficient PDE operator $\A$ ($k \in \N$), i.e.\
\begin{equation} \label{eq:muPDE}
  \A \mu := \sum_{|\alpha|\le k}A_{\alpha} \partial^\alpha \mu = 0  \qquad \textrm{in \(\D'(\Omega;\R^n)\).}
\end{equation}
Here, $A_{\alpha}\in \R^{n\times m}$ and $\partial^\alpha=\partial_1^{\alpha_1}\ldots\partial_d^{\alpha_d}$ for each multi-index $\alpha=(\alpha_1,\ldots,\alpha_d)\in (\N \cup \{0\})^d$.    A central question about~\eqref{eq:muPDE} asks what can be said about the \emph{singular part} $\mu^s$ of solutions $\mu = g \L^d + \mu^s$ ($\mu^s \perp \L^d$). Besides Alberti's celebrated rank-one theorem~\cite{Alberti93} for $\A = \curl$, not much is known at present.

In this respect we recall that the \emph{wave cone}
\[
  \Lambda_\A\defeq\bigcup_{|\xi|=1} \Ker \AA^k(\xi) \subset \R^m \qquad\textrm{with}\qquad \AA^k(\xi)\defeq (2\pi \ii)^{k} \sum_{|\alpha|=k}A_{\alpha}\xi^{\alpha},
\]
where $\xi^\alpha = \xi_1^{\alpha_1} \cdots \xi_d^{\alpha_d}$, plays a crucial role in the compensated compactness theory for sequences of $\A$-free maps~\cite{Murat78, Murat79, Tartar79, Tartar83, DiPerna85, Rindler14MCF}. Indeed, \(\Lambda_\A\)  contains the values that an oscillating or concentrating sequence of functions is expected to attain. The corresponding \emph{characteristic} $\xi$'s determine the allowed directions of oscillations and concentrations.

Since the singular part $\mu^s$ of a measure contains ``condensed'' oscillations and concentrations, it is natural to conjecture that for a measure $\mu$ solving~\eqref{eq:muPDE}, the \emph{polar} $\frac{\d\mu}{\d|\mu|}$, i.e.\ the Radon--Nikod\'{y}m derivative of $\mu$ with respect to its total variation measure $|\mu|$, must lie in the wave cone at almost all singular points. For $\A = \curl$ this was conjectured by Ambrosio \& De Giorgi in~\cite{AmbrosioDeGiorgi88} and proved by Alberti in~\cite{Alberti93}. Our main result asserts the truth of this conjecture in full generality:

\begin{theorem}\label{main}
Let \(\Omega\subset \R^d\) be an open set, let \(\A\) be a $k$'th-order linear constant-coefficient differential operator as above, and let \(\mu \in \M(\Omega;\R^m)\) be an \(\A\)-free Radon measure on $\Omega$ with values in $\R^m$. Then,
\[
\frac{\d\mu}{\d|\mu|}(x)\in \Lambda_\A \qquad\textrm{for \(|\mu|^s\)-a.e.\ \(x\in \Omega\).}
\]
\end{theorem}

\begin{remark}
Note that (perhaps surprisingly) we do not need to require $\A$ to satisfy Murat's constant-rank condition~\cite{Murat81}.
\end{remark}

\begin{remark}\label{rmk:rhs}
Let us point out that  Theorem~\ref{main} is also valid in  the situation
\begin{equation}\label{eq:rhs}
\A\mu =\sigma \quad\text{for some}\quad \sigma\in \M(\Omega;\R^n).
\end{equation}
This can be reduced to the setting of Theorem~\ref{main} by defining  \(\tilde \mu=(\mu,\sigma)\in \M(\R^d;\R^{m+n})\) and  \(\tilde \A\) (with an additional $0$'th-order term) such that \eqref{eq:rhs} is equivalent to \(\tilde {\A} \tilde \mu=0\). It is easy to check that, if \(k\ge1\), \(\Lambda_{\tilde{\A}} = \Lambda_\A \times \R^n\) and that for \(|\mu|\)-almost every point \(\frac{\d \mu} {\d |{\mu}|}\) is proportional to \(\frac{\d \mu} {\d |\tilde{\mu}|}\).
\end{remark}

\begin{remark}\label{rmk:constant}
Using essentially the same proof, Theorem~\ref{main} can be further extended to the setting of \emph{variable-coefficient} linear differential operators \(\A=\sum_\alpha A_\alpha(x)\partial^\alpha\) with the coefficients satisfying suitable regularity assumptions.  In this setting, the conclusion reads
\[
\frac{\d\mu}{\d|\mu|}(x)\in \Lambda_\A(x):=\bigcup_{|\xi|=1} \Ker \AA_{x}^k(\xi) \qquad\textrm{for \(|\mu|^s\)-a.e.\ \(x\),}
\]
where
\[
\AA_{x}^k(\xi)\defeq\sum_{|\alpha|=k}(2\pi \ii)^{k} A_{\alpha}(x)\xi^{\alpha}.
\]
Similar statements can be obtained if \(\mu\) solves some pseudo-differential equations.
\end{remark}

By applying  Theorem~\ref{main} to suitably chosen differential operators, we easily obtain several remarkable consequences, which are outlined below.
 In particular, we provide a simple proof of Alberti's rank-one theorem and, for the first time,  its extensions to functions of bounded deformation (BD).  We also prove  a structure theorem for  the singular part of a  finite family of normal currents in the spirit of the rank-one theorem. By relying on the results of Alberti \& Marchese~\cite{AlbertiMarchese15} and of Schioppa~\cite{Schioppa15},  the latter result immediately implies that  the Rademacher theorem can hold only for absolutely continuous measures and that every top-dimensional Ambrosio--Kirchheim metric current in \(\R^d\) is a Federer--Fleming flat chain (a part of the so-called ``flat chain conjecture'', see~\cite[Section~11]{AmbrosioKirchheim00}).

\subsection{Rank-one property of BV-derivatives}

As already mentioned above, in~\cite{Alberti93} Alberti solved a conjecture of Ambrosio \& De Giorgi~\cite{AmbrosioDeGiorgi88} by showing the  rank-one property for the singular part of the gradients of  \(\BV\)-functions (also see~\cite{DeLellis08,AlbertiCsornyeiPreiss05}). Besides its theoretical interest, the rank-one theorem has many applications in the theory of functions of bounded variation, we just mention the following: lower-semicontinuity and relaxation~\cite{AmbrosioDalMaso92,FonsecaMuller93,KristensenRindler10Relax}, integral representation theorems~\cite{BouchitteeFonsecaMascarenhas98}, Young measure theory~\cite{KristensenRindler10YM, Rindler14YM, KirchheimKristensen16}, approximation theory~\cite{KristensenRindler15}, and the study of continuity equations with BV-vector fields~\cite{Ambrosio04} (in the latter case  the use of the rank-one theorems can however be avoided, see~\cite[Remark 3.7]{Ambrosio04} and~\cite{Ambrosio08}). We refer to~\cite[Chapter~5]{AmbrosioFuscoPallara00} for further history.

\begin{theorem}[Alberti's rank-one theorem]\label{corBV}
Let \(\Omega\subset \R^d\) be an open set and let  \(u\in \BV(\Omega; \R^\ell)\). Then, for \(|D^s u|\)-almost every \(x\in \Omega\), there exist $a(x)\in \R^\ell \setminus \{0\}$, $b(x)\in \R^d \setminus \{0\}$ such that
\[
\frac{\d D^s u}{\d|D^s u|}(x)=a(x)\otimes b(x).
\]
\end{theorem}

Alberti's rank-one theorem easily follows by  choosing  \(\A=\curl\) in Theorem~\ref{main}. Let us also mention that Massaccesi \& Vittone have recently given a short and elegant proof of the rank-one property based on the theory of sets of finite perimeter~\cite{MassaccesiVittone16}.

As already observed by Alberti in~\cite[Theorem 4.13]{Alberti93}, Theorem~\ref{corBV} implies the validity of a similar property for higher-order derivatives. A direct proof of this fact can also be obtained as a corollary of   our Theorem~\ref{main}:

\begin{theorem}[Rank-one theorem for higher-order derivatives]\label{corHigher}
Let $\Omega\subset \R^d$ be an open set and let $u \in \Lrm^1(\Omega;\R^\ell)$ with $ D^r u \in \M(\Omega;\mathrm{SLin}^r(\R^d;\R^\ell))$ for some $r \in \N$, where $\mathrm{SLin}^r(\R^d;\R^\ell)$ contains all symmetric $r$-linear maps from $\R^d$ to $\R^\ell$. Then, for $|(D^r u)^s|$-almost every $x\in \Omega$, there exist $a(x)\in \R^\ell \setminus \{0\}$, $b(x)\in \R^d \setminus \{0\}$ such that
\[
\frac{\d (D^r u)^s}{\d|(D^r u)^s|}(x) = a(x)\otimes \underbrace{b(x) \otimes \cdots \otimes b(x)}_{\textrm{\(r\) times}}.
\]
\end{theorem}

\subsection{Polar density theorem for BD-functions}

The proofs in~\cite{Alberti93} and in~\cite{MassaccesiVittone16} of Alberti's rank-one theorem strongly rely on the structure of functions of bounded variation and on their link with the theory of sets of finite perimeter. In particular, so far it has remained open whether a similar statement is valid for the larger class of functions of {\em bounded deformation}, i.e.\ those functions \(u\in \Lrm^1(\Omega;\R^d)\)  whose symmetric part of the (distributional) derivative is  a measure,
\[
Eu\defeq \frac{Du+(Du)^T}{2}\in \M(\Omega;\R^{d \times d}_{\mathrm{sym}}).
\]
We collect all these functions into the set $\BD(\Omega)$; see~\cite{TemamStrang80,Temam85book,AmbrosioCosciaDalMaso97} for a detailed account of the theory of this space. 

The extension of Alberti's rank-one theorem to the space of functions of bounded deformation  follows from our main Theorem~\ref{main}  with the appropriate choice of  the differential operator \(\A\):

\begin{theorem}\label{corBD}
Let \(\Omega\subset \R^d\) be an open set and let \(u\in \BD(\Omega)\). Then, for \(|E^s u|\)-almost every \(x\in \Omega\), there exist $a(x), b(x)\in \R^d \setminus \{0\}$ such that
\[
\frac{\d E^s u}{\d| E^s u|}(x)=a(x)\odot b(x),
\]
where we define the symmetrized tensor product as  \(a\odot b:=(a\otimes b+b\otimes a)/2\) for \(a,b\in \R^d\).
\end{theorem}

This theorem has consequences for the structure theory of BD-functions and lower semicontinuity theory (in the lower semicontinuity theory our structure theorem can, however, be avoided at the price of some mild restrictions on the functional, see~\cite{Rindler11} for BD and~\cite{Rindler12} for an analogous result in BV); some of these consequences will be explored in future work.

Further, in~\cite{Temam85book,FuchsSeregin00book,ContiOrtiz05} it is motivated why the space
\[
  \Urm(\Omega) := \setb{ u \in \BD(\Omega) }{ \diverg u \in \Lrm^2(\Omega) }
\]
is the appropriate space for elasto-plasticity theory in the geometrically linear setting. For this space we immediately get the following structure result:

\begin{corollary} \label{corU}
Let $\Omega\subset \R^d$ be an open set and let $u\in \Urm(\Omega)$. Then, for $\abs{E^s u}$-almost every $x \in \Omega$, there exist $a(x), b(x) \in \R^d \setminus \{0\}$ with
\[
  a(x) \perp b(x)
\]
such that
\[
 \frac{\di E^s u}{\di \abs{E^s u}}(x) = a(x) \odot b(x).
\]
\end{corollary}

\subsection{Normal currents, the Rademacher theorem, and metric currents}\label{intro:subseccurr}

Our next application of Theorem~\ref{main} deals with finite families of (Euclidean) normal currents, by which we obtain some  consequences concerning differentiability of Lipschitz functions and the theory of metric currents. We assume the reader to be  familiar with the theory of currents and with basic  multilinear algebra. We refer to~\cite[Chapters~1~\&~4]{Federer69} and Section~\ref{appl} below for the relevant notations and definitions.

To motivate our next  result,  recall that any \((d-1)\)-dimensional  normal current \(T\in \mathbf{N}_{d-1}(\R^d)\) without boundary (\(\partial T=0\)) can be identified via Hodge duality with the derivative of a  function \(u\in \BV_{\rm loc}(\R^d;\R)\), that is, \(T=\star D u\). Using this identification and the fact that \(\dim \Lambda_{d-1} (V)=1\) if and only if  \(\dim (V)=d-1\), Theorem~\ref{corBV} can be rephrased as follows.

\begin{corollary}
Let   \(T_1=\vec{T}_1\|T_1\|,\ldots,T_r=\vec{T}_r\|T_r\|\in \mathbf{N}_{d-1}(\R^d)\)  be \((d-1)\)-dimensional boundaryless normal currents, i.e.\ \(\partial T_i=0\) for \(i=1,\ldots,r\).  Let further \(\mu\in \M_+(\R^d)\) be a positive Radon measure such that
\[
  \mu\aac \|T_i\| \qquad
  \text{for \(i=1,\ldots,r\).}
\]
Then, for  \(\mu^s\)-a.e.\  \(x\in\R^d\) there exists a \((d-1)\)-dimensional subspace \(V_x \subset \R^d\)  such that \(\vec{T}_1(x),\ldots,\vec{T}_r(x)\in \Lambda_{d-1}(V_x)\). 
\end{corollary}

As another simple application of Theorem~\ref{main} we can generalize the above statement to finite families of  normal currents (not  necessarily of the same dimension).

\begin{theorem}\label{thm:curr}
Let $\Omega \subset \R^d$ be an open set and let \(T_1=\vec{T}_1\|T_1\|\in  \mathbf{N}_{k_1}(\Omega), \ldots \,, T_r=\vec{T}_r\|T_r\|\in \mathbf{N}_{k_r}(\Omega)\) be normal currents, where \(k_1,\ldots,k_r\in \{1,\ldots,d\}\), \(r\in \N\). Let further \(\mu\in \M_+(\Omega)\) be a positive Radon measure such that
\[
  \mu\aac \|T_i\| \qquad
  \text{for \(i=1,\ldots,r\).}
\]
Then, for  \(\mu^s\)-a.e.\  \(x\in \Omega\) there exists a \(1\)-covector \(\omega_x\in \Lambda^1(\R^d) \setminus \{0\}\) such that 
\[
\vec{T}_1(x)\imult \omega_x=\ldots=\vec{T}_r(x)\imult \omega_x=0.
\]
Equivalently, for  \(\mu^s\)-a.e.\ \(x \in \Omega\), \(\vec{T}_1(x)\in \Lambda_{k_1}(\Ker \omega_x),\ldots,\vec{T}_r(x)\in \Lambda_{k_r}(\Ker \omega_x)\).
 \end{theorem}

\begin{remark}\label{rmk:curios}
Let us note in passing the following curious consequence of the above result: It is well known that, apart from the trivial cases \(k\in \{1,d-1,d\}\), the orienting vector \(\vec T\) of a \(k\)-dimensional normal current \(T\) need not be simple, i.e.\ of the form \(\vec T(x)=v_1(x)\wedge\ldots\wedge v_k(x)\), \(v_i(x)\in \R^d\). However, if \(\dim V=(d-1)\), then  every \(w\in \Lambda_{d-2}(V)\) is necessarily simple. Thus, we have that for \(T\in \mathbf{N}_{d-2}^{\textrm{loc}}(\R^d)\) the simplicity of $\vec{T}$ holds \(\|T\|^s\)-almost everywhere. Note that the current
 \[
 T=(e_1\wedge e_2+e_3\wedge e_4)\,\H^4\res\{x_5=0\}\in \mathbf N_{2}^{\textrm{loc}} (\R^5),
 \]
shows that this statement is false for $k$-dimensional currents with \(1<k<(d-2)\).
\end{remark}

A particularly relevant instance  of  Theorem~\ref{thm:curr}  is obtained when  \(r=d\) and  \(k_1=\ldots=k_d=1\). In view of the  subsequent applications, let us  state it in a slightly different (but equivalent) formulation:

\begin{corollary}\label{cor:1curr}
Let \(T_1=\vec{T}_1 \|T_1\|,\ldots,T_d=\vec{T}_d \|T_d\|\in \mathbf{N}_1(\R^d)\) be one-dimensional normal currents  such that there exists a positive Radon measure \(\mu\in \M_+(\R^d)\) with the following properties:
\begin{itemize}
\item[(i)]  \(\mu\aac \|T_i\|\) for \(i=1,\ldots,d\), 
\item[(ii)] for \(\mu\)-almost every \(x \), \(\Span\{\vec {T}_1(x),\ldots,\vec{T}_d(x)\}=\R^d\).
\end{itemize}
Then, \(\mu\aac \L^{d}\).
\end{corollary}

This answers the question about a higher-dimensional analogue of~\cite[Proposition~8.6]{AlbertiCsornyeiPreiss05}. By the trivial identification  of one-dimensional normal currents with vector-valued measures, Corollary~\ref{cor:1curr} can be stated in the following equivalent formulation, which in a sense is dual to Theorem~\ref{corBV}. It can be also directly inferred from Theorem~\ref{main}.

\begin{corollary}\label{corDiv}
Let \(\mu\in \M(\Omega;\R^{d\times d})\) be a matrix-valued measure such that
\[
  \Div \mu \in \M(\Omega; \R^d).
\] 
Then,
\[
{\rm rank } \biggl(\frac{\d \mu}{\d |\mu|}(x) \biggr) \le d-1  \qquad
\text{for \(|\mu|^s\)-a.e.\ \(x\in \Omega\).}
\]
\end{corollary}

It has been noted in several places that the validity of the rank-one theorem for maps \(u\in \BV(\R^2;\R^2)\) has some direct implications concerning differentiability of Lipschitz functions and the structure of top-dimensional metric currents in the plane~\cite{Preiss90,AlbertiCsornyeiPreiss05,AlbertiCsornyeiPreiss10,PreissSpeight14,AlbertiMarchese15, Schioppa15}. Relying on~\cite{AlbertiMarchese15, Schioppa15},  we use Corollary~\ref{cor:1curr}  to extend these results to every dimension. In particular, Theorem~\ref{cor:curr} below  provides a positive answer to the case \(k=d\) of the  ``flat chain conjecture'' stated in~\cite[Section 11]{AmbrosioKirchheim00}, see~\cite[Theorem 1.6]{Schioppa15} for the case \(k=1\).

\begin{theorem}\label{cor:rad}
Let \(\mu\in \M_+(\R^d)\) be a positive Radon measure such that every Lipschitz map \(f:\R^d\to \R\) is differentiable \(\mu\)-almost everywhere. Then, \(\mu\aac\L^d\).
\end{theorem}

\begin{theorem}\label{cor:curr}
Let \(T\in \mathbf M^{\rm met}_d(\R^d)\) be an Ambrosio--Kirchheim metric current of  dimension $d$, see~\cite{AmbrosioKirchheim00}. Then, \(\|T\|\aac \L^d\).  In particular, the space of \(d\)-dimensional metric currents in \(\R^d\) coincides with the space of  Federer--Fleming \(d\)-dimensional flat chains, \(\mathbf M^{\rm met}_d(\R^d)=\mathbf F_d(\R^d)\).
\end{theorem}

Let us mention  that the last two theorems will also follow by a stronger result announced by Cs\"ornyei and Jones in~\cite{Jones11talk}, namely that for every Lebesgue null set \(E\subset \R^d\) there exists a Lipschitz map \(f:\R^d\to \R^d\) which is nowhere differentiable in \(E\), see the discussion in the introduction of~\cite{AlbertiMarchese15} for a detailed account of these type of results.

\subsection{Sketch of the proof}
We conclude this introduction with an outline of the main ideas behind the  proof of  Theorem~\ref{main}. Let us assume for simplicity that \(\A\)  is a first-order homogeneous operator, \(\A=\sum_\ell A_\ell \partial_\ell\).  Assume by contradiction that there is a set  \(E\) of  positive \(|\mu|^s\)-measure such that the polar vector \( \frac{\d \mu}{\d|\mu|}(x)\) is not in the wave cone \(\Lambda_\A\)  for every \(x\in E\). One can then find a point \(x_0\in E\) and a sequence \(r_j\downarrow0\) such that
\[
\wslim_{j\to \infty} \frac{(T^{x_0,r_j})_\sharp\mu}{|\mu|(B_{r_j}(x_0))}=\wslim_{j\to \infty} \frac{(T^{x_0,r_j})_\sharp\mu^s}{|\mu|^s(B_{r_j}(x_0))}=P_0 \nu,
\]
where  \(T^{x,r}:\R^d\to \R^d\) is the dilation map \(T^{x,r}(y)=(y-x)/r\), \(T^{x,r}_\sharp\) denotes the push-forward operator (that is, for any measure \(\sigma\) and Borel set \(B\), \([(T^{x,r})_\sharp\sigma](B):=\sigma(x+rB)\)), \(\nu\in {\rm Tan} (x_0,|\mu|)={\rm Tan} (x_0,|\mu|^s)\) is a non-zero tangent measure in the sense of Preiss~\cite{Preiss87}, and 
\[
P_0\defeq \frac{\d \mu}{\d|\mu|}(x_0)\notin \Lambda_\A.
\]
Moreover, one easily checks that
\[
\sum_{\ell=1}^d A_\ell P_0 \,\partial_\ell \nu=0  
\qquad \textrm{in \(\D'(\Omega;\R^n)\).}
\]
By taking the Fourier transform of the above equation, we get
\[
\AA(\xi)P_0 \, \hat \nu(\xi)=0,  \qquad \xi \in \R^d.
\]
Having assumed that \(P_0\notin \Lambda_\A\), this implies \({\rm supp}\,\hat \nu=\{0\}\) and thus \(\nu\aac \L^d\).  The latter fact, however, is not by itself a contradiction to 
\(\nu \in {\rm Tan} (x_0,|\mu|^s)\). Indeed, Preiss~\cite{Preiss87}  provided an example of a purely singular measure that has only multiples of Lebesgue measure as tangents (we also refer to~\cite{Oneill95} for a measure that has \emph{every} measure as a tangent at almost every point).

On the other hand, \(P_0\notin \Lambda_\A\) implies that  \(\AA(\xi)P_0\ne 0\), so one can hope for some sort of ``elliptic regularization'' that forces not only \(\nu\aac \L^d\) but also \(|\mu|^s\aac \L^d\) in a neighborhood of \(x_0\). In fact, this is (almost) the case: Inspired by Allard's Strong Constancy Lemma in~\cite{Allard86} and using some basic pseudo-differential calculus,  we can show that in the above situation not only 
\[
\nu_j\defeq  \frac{(T^{x_0,r_j})_\sharp\mu^s}{|\mu|^s(B_{r_j}(x_0))}\weakstarto \nu\aac \L^d
\]
but that, crucially, this convergence also holds in the total variation norm,
\[
  |\nu_j-\nu|(B_1)\to 0.
\]
  Since \(\nu_j\perp \L^d\), this latter fact easily gives a contradiction to $\nu \ll \L^d$ and concludes the proof of the theorem.

\subsection*{Acknowledgments} 

The authors would like to thank  A.~Massaccesi and D.~Vittone for useful discussions. G.~D.~P.\ is supported by the MIUR SIR-grant ``Geometric Variational Problems" (RBSI14RVEZ) and F.~R.\ acknowledges the support from an EPSRC Research Fellowship on ``Singularities in Nonlinear PDEs'' (EP/L018934/1).

\section{Proof of the main theorem}

\subsection{Notation} 

We denote by $\M(\Omega;\R^m)$ the space of all finite Radon measures on an open set $\Omega \subset \R^d$ with values in $\R^m$ and by \(\M_+(\Omega)\) the space of positive Radon measures on $\Omega$. We write \(\mu=\wslim_{j\to \infty} \mu_j\) or \(\mu_j \weakstarto \mu\) for the local weak*-convergence of $\mu_j$ to $\mu$, that is \(\int \varphi \,\di \mu_j \to \int \varphi \,\di \mu\) for all \(\varphi \in \Crm^0_c(\Omega)\), the set of all continuous functions with compact support in $\Omega$. The $d$-dimensional Lebesgue measure is  \(\L^d\). Given a Borel set \(B \subset \Omega\) and a measure \(\mu \in \M(\Omega;\R^m)\) (or \(\mu\in \M_+(\Omega)\)), we denote by \(\mu\res B\)  the restriction of \(\mu\) to \(B\).

The Lebesgue--Radon--Nikod\'{y}m decomposition of a Radon measure $\mu \in \M(\Omega;\R^m)$ is given as
\[
\mu=\frac{\d \mu } {\d |\mu|} |\mu| = \mu^a + \mu^s = g \L^d + \frac{\d \mu } {\d |\mu|} |\mu|^s,
\]
where $\frac{\d \mu } {\d |\mu|} \in \Lrm(\Omega,|\mu|;\R^m)$ is the polar of $\mu$, i.e.\ the Radon--Nikod\'{y}m derivative of $\mu$ with respect to $\mu$'s total variation measure \(|\mu| \in \M_+(\Omega)\), $\mu^a \ll \L^d$ is the absolutely continuous part of $\mu$ with density \(g \in \Lrm^1(\Omega)\), and $\mu^s \perp \L^d$ is the singular part of $\mu$. Note that here and in the following the terms ``singular'' and ``absolutely continuous'' are always understood with respect to the Lebesgue measure if not otherwise specified.

We will generically denote by   \(\A\) a $k$'th-order linear partial differential operator with constant coefficients that acts on smooth functions \(u\in \Crm^\infty(\R^d; \R^m)\) as 
\begin{equation*}\label{ufree}
\A u := \sum_{|\alpha|\le k}A_{\alpha} \partial^\alpha  u  \in \Crm^\infty(\R^d;\R^n),
\end{equation*}
where \(\alpha=(\alpha_1,\ldots,\alpha_d)\in (\N \cup \{0\})^d\) is a multi-index, \(\partial^\alpha=\partial_1^{\alpha_1}\ldots\partial_d^{\alpha_d}\), and   \(A_{\alpha}\in \R^{n\times m}\) are matrices. A vector-valued Radon measure \(\mu\in \M(\Omega;\R^m)\) is said to be \(\A\)-free if 
\begin{equation*}\label{mufree}
\A \mu=0\qquad \textrm{in \(\D'(\Omega;\R^n)\).}
\end{equation*}
Here, \(\D(\Omega;\R^n) = \Crm^\infty_c(\Omega;\R^n)\) is the set of \(\R^n\)-valued test functions in \(\Omega\) with the usual topology and \(\D'(\Omega;\R^n)\) is the set of \(\R^n\)-valued distributions on \(\Omega\).

Given \(\A\) as above, its \emph{symbol} $\AA \colon \R^d \to \R^{n \times m}$ is defined as
\[
\AA(\xi)\defeq\sum_{|\alpha|\le k}(2\pi i)^{|\alpha|}\, A_{\alpha}\xi^{\alpha},  \qquad
\xi \in \R^d,
\] 
where \(\xi^{\alpha}:=\xi_1^{\alpha_1}\ldots \xi_d^{\alpha_d}\). Note that for \(u\) in the Schwartz class \(\mathcal S(\R^d;\R^m)\),
\begin{equation*}
\hat{ \A u}(\xi)=\AA(\xi)\hat u(\xi),
\end{equation*}
where for  \(v\in \S(\R^d;\R^m)\)  we denote by   \(\hat v \) its  Fourier transform,
\[
\hat v(\xi)=\F[v](\xi):=\int v(x)\ee^{-2\pi \ii\, x\cdot \xi} \dd x,  \qquad \xi \in \R^d.
\]
We also  recall the definition  of the {\em wave cone} associated to \(\A\)~\cite{Tartar79,Murat81,Tartar83,DiPerna85}:
\begin{equation*}\label{wc}
\Lambda_\A:=\bigcup_{|\xi|=1} \Ker \AA^k(\xi) \subset \R^m \qquad\textrm{with}\qquad \AA^k(\xi)\defeq (2\pi i)^{k} \sum_{|\alpha|=k}A_{\alpha}\xi^{\alpha}.
\end{equation*}

\subsection{First-order operators}

For the sake of illustration, we first  treat  the case when  \(\A\) is  a  first-order homogeneous constant-coefficient differential operator, namely
\begin{equation}\label{firstorder}
\A\mu =\sum_{\ell=1}^d A_\ell \partial_\ell \mu = 0  \qquad \textrm{in \(\D'(\Omega;\R^n)\).}
\end{equation}

\begin{proof}[Proof of Theorem~\ref{main} assuming~\eqref{firstorder}]
We have
\[
\Lambda_\A=\bigcup_{|\xi|=1}\Ker \AA(\xi),
\qquad \AA(\xi)=\AA^1(\xi) =2\pi \ii\, \sum_{\ell=1}^d A_\ell \xi_\ell.
\]
Let 
\[
E\defeq\setBB{x\in \Omega}{ \frac{\d\mu}{\d|\mu|}(x)\notin \Lambda_\A },
\]
where the existence of $\frac{\d\mu}{\d|\mu|}(x)$ in the sense of the Besicovitch derivation theorem, see~\cite[Theorem~2.22]{AmbrosioFuscoPallara00}, is part of the definition of $E$.

Assume by contradiction that \(|\mu|^s(E)>0\). We  now choose a point \(x_0\in E\) and a sequence \(r_j\downarrow 0\) such that
\begin{itemize}
\item[(i)] $\displaystyle \lim_{j\to \infty} \frac{|\mu|^a(B_{r_j}(x_0))}{|\mu|^s(B_{r_j}(x_0))}=0$ and $\displaystyle \lim_{j\to \infty}  \mean{B_{r_j}(x_0)} \biggl|\frac{\d \mu}{\d|\mu|}(x)-\frac{\d \mu}{\d|\mu|}(x_0)\biggr|\dd|\mu|^s(x)=0;
$

\item[(ii)] there exists a positive Radon measure \(\nu\in \M_+(\R^d) \) with $\nu \res B_{1/2} \neq 0$ and such that 
\[
\nu_j\defeq \frac{(T^{x_0,r_j})_{\sharp}|\mu|^s}{|\mu|^s(B_{r_j}(x_0))}\weakstarto \nu;
\]
\item[(iii)] for the polar vector it holds that
\[
P_0 \defeq \frac{\d \mu}{\d|\mu|}(x_0)\notin \Lambda_\A
\]
and there is a positive constant \(c>0\) such that \(|\AA(\xi)P_0|\ge c|\xi|\) for \(\xi \in \R^d\).
\end{itemize}
Indeed, (i) holds at \(|\mu|^s\)-almost every point by classical measure theory, (ii) follows by the fact that for \(|\mu|^s\)-almost every \(x \in \Omega\) the space of tangent measures \({\rm Tan}(|\mu|^s,x)\) to \(|\mu|^s\) at \(x\) is non-trivial, see for instance~\cite[Theorem~2.5]{Preiss87} or~\cite[Lemma A.1]{Rindler11}, and finally, (iii) follows from the assumption \(|\mu|^s(E)>0\).

We now claim that  (i)--(iii) above imply that
\begin{gather}
0\ne \nu\res B_{1/2} \aac \mathcal L^d,  \label{tan1}\\
\lim_{j\to  \infty} |\nu_j-\nu|(B_{1/2}) =0. \label{tan2}
\end{gather}
Before proving \eqref{tan1} and \eqref{tan2}, let us show how to use them to conclude  the proof. Recall that \(\nu_j\perp \mathcal L^d\) and take Borel sets \(E_j \subset B_{1/2}\) with \(\L^d(E_j)=0=\nu(E_j)\) and \(\nu_j(E_j)=\nu_j(B_{1/2})\). Then, 
\[
\nu_j(B_{1/2})=\nu_j(E_j)\le |\nu_j-\nu|(B_{1/2})+\nu(E_j)=|\nu_j-\nu|(B_{1/2})\to 0,
\]
thanks to \eqref{tan2}. Hence, we infer $\nu(B_{1/2}) = 0$, in contradiction to~\eqref{tan1}. Thus, \(|\mu|^s(E)=0\), concluding the proof of the theorem.

We are thus left to prove \eqref{tan1} and \eqref{tan2}.  Let us assume that \(x_0=0\) and set \(T^r:=T^{x_0,r}\).
Clearly,
\[
  \A\big(T^r_\sharp  \mu\big)=0  \qquad\text{in \(\D'(B_1;\R^n)\).}
\]
Therefore, with \(\nu_j\) defined as in (ii) above and \(c_j:=|\mu|^s(B_{r_j})^{-1}\),
\begin{equation}\label{pezzi2}
\A(P_0\nu_j)= \A(P_0\nu_j-c_jT^{r_j}_\sharp\mu).
\end{equation}
Let now \(\{\varphi_\eps\}_{\eps>0}\) be a compactly supported, smooth, and positive approximation of the identity. By the lower semicontinuity of the total variation,
\[
|\nu_j-\nu|(B_{1/2})\le \liminf_{\eps\to 0}\, |\nu_j* \varphi_{\eps}-\nu|(B_{1/2}).
\]
Thus, for every \(j\) we can find \(\eps_j\le 1/j\) such that
\begin{equation}\label{quasi}
|\nu_j-\nu|(B_{1/2})\le |\nu_j* \varphi_{\eps_j}-\nu |(B_{1/2})+\frac{1}{j}.
\end{equation}
We now convolve \eqref{pezzi2} with \(\varphi_{\eps_j}\) to get
\begin{equation}\label{pezzi3}
\A(P_0 u_j)= \A(V_j),
\end{equation}
where we have set
\begin{equation*}\label{set}
u_j\defeq \nu_j* \varphi_{\eps_j}, \qquad V_j\defeq \big[P_0\nu_j-c_jT^{r_j}_\sharp\mu \big]* \varphi_{\eps_j}.
\end{equation*}
Note that \(u_j, V_j\) are smooth, \(u_j\ge 0\), and
\begin{equation}\label{conv}
u_j \weakstarto \nu.
\end{equation}	
Moreover, recalling that \(x_0=0\) and \(c_j=|\mu|^s(B_{r_j})^{-1}\), by the definition of  \(V_j\), \(\nu_j\), \(P_0\) and standard properties of convolutions, see~\cite[Theorem 2.2]{AmbrosioFuscoPallara00},  for \(\eps_j\le 1/4\) it holds that
\begin{equation*}
\begin{split}
\int_{B_{3/4}} |V_j| \dd x&\le
  \frac{\big|P_0\, T^{r_j}_\sharp|\mu|^s-T^{r_j}_\sharp\mu \big|(B_1)}{|\mu|^s(B_{r_j})}\\
  &\le  \frac{ \big|P_0\, |\mu|^s-\mu^s \big|(B_{r_j})}{|\mu|^s(B_{r_j})} + \frac{|\mu|^a(B_{r_j})}{|\mu|^s(B_{r_j})} \\
 &= \mean{B_{r_j}} \biggl|\frac{\d \mu}{\d|\mu|}(0)-\frac{\d \mu}{\d|\mu|}(x)\biggr|\d|\mu|^s(x)+\frac{|\mu|^a(B_{r_j})}{|\mu|^s(B_{r_j})}.
\end{split}
\end{equation*}
Hence, by (i) above,
\begin{equation}\label{infinitesimal}
 \lim_{j\to \infty}\int_{B_{3/4}} |V_j| \dd x =0.
\end{equation}

Take a cut-off function \(\chi\in \D(B_{3/4})\) with  \(0\le \chi\le 1\) and \(\chi\equiv1\) on \(B_{1/2}\). Then, \eqref{pezzi3} implies that 
\begin{equation}\label{pezzi4}
\A(P_0 \chi u_j)= \chi  \A(P_0 u_j) + \A(P_0 \chi)u_j = \A(\chi V_j)+R_j,
\end{equation}
where the remainder terms \( R_j\defeq \A(P_0\chi)u_j -\sum_{\ell} A_\ell V_j \partial_{\ell}\chi  \) are smooth, compactly supported in \(B_1\), and  satisfy
\[
\sup_{j} \int_{B_1} |R_j| \dd x \le C
\]
for some constant \(C\) thanks to  \eqref{conv} and \eqref{infinitesimal}. Taking the Fourier transform of \eqref{pezzi4}, we obtain
\begin{equation*}
[\AA(\xi)P_0] \, \hat {\chi u_j}(\xi)= \AA(\xi) \hat {\chi V_j}(\xi)+\hat{R}_j(\xi).
\end{equation*}
Now multiply  by  \([\AA(\xi)P_0]^* = \overline{[\AA(\xi)P_0]^T}\) and add \(\hat {\chi u_j}(\xi)\) to both sides of the above equation to obtain 
\begin{equation*}
(1+|\AA(\xi)P_0|^2) \, \hat {\chi u_j}(\xi)= [\AA(\xi)P_0]^* \AA(\xi) \, \hat {\chi V_j}(\xi)+ \hat {\chi u_j}(\xi) + [\AA(\xi)P_0]^*  \hat{R}_j(\xi),
\end{equation*}
which can be rewritten as 
\begin{equation*}
\begin{split}
 \hat {\chi u_j}(\xi)&= \frac{[\AA(\xi)P_0]^* \AA(\xi) \, \hat {\chi V_j}(\xi)}{1+|\AA(\xi)P_0|^2}
  + \frac{1+4\pi^2|\xi|^2}{1+|\AA(\xi)P_0|^2} \cdot \frac{\hat {\chi u_j}(\xi)}{1+4\pi^2|\xi|^2}\\
 &\qquad +\frac{(1+4\pi^2|\xi|^2)^{1/2} [\AA(\xi)P_0]^*}{1+|\AA(\xi)P_0|^2}\cdot \frac{\hat{R}_j(\xi)}{(1+4\pi^2|\xi|^2)^{1/2}}.
\end{split}
\end{equation*}
Hence,
\begin{equation}\label{pezzi5}
\chi u_j=T_0 [\chi V_j]+T_1[\chi u_j]+T_2[R_j] =: f_j+g_j+h_j
\end{equation}
with
\begin{align*}
T_0[V] &:= \F^{-1} \Bigl[(1+|\AA(\xi)P_0|^2)^{-1}[\AA(\xi)P_0]^* \AA(\xi)\hat {V}(\xi)\Bigr],\\
T_1[u]&:=\F^{-1} \Bigl[m_1(\xi)(1+4\pi^2|\xi|^2)^{-1}\hat {u}(\xi)\Bigr],\\
T_2[R]&:=\F^{-1} \Bigl[m_2(\xi)(1+4\pi^2|\xi|^2)^{-1/2} \hat {R}(\xi)\Bigr],
\end{align*}
where we have set
\begin{align*}
m_1(\xi)&=(1+|\AA(\xi)P_0|^2)^{-1}(1+4\pi^2|\xi|^2),\\
m_2(\xi)&=(1+|\AA(\xi)P_0|^2)^{-1}(1+4\pi^2|\xi|^2)^{1/2}[\AA(\xi)P_0]^*.
\end{align*}

By (iii) above, \(T_0\) is an operator associated with an H\"ormander--Mihlin multiplier (meaning that it has a smooth  symbol \(m_0(\xi)\) such that \(|\partial^\beta m_0(\xi)| \leq K|\xi|^{-|\beta|}\) for every multi-index \(|\beta|\le \lfloor d/2 \rfloor + 1 \) and some $K > 0$).   The  $\Lrm^1$--$\Lrm^{1,\infty}$ estimates~\cite[Theorem 5.2.7]{Grafakos14book1} in conjunction with~\eqref{infinitesimal} give
\begin{equation}\label{eq:1}
\sup_{\lambda\ge 0}\, \lambda\, \L^d \bigl( \{|f_j|>\lambda\}\bigr)\le C \|\chi V_j\|_{\Lrm^1}\to 0.
\end{equation}
Moreover, the operators  \(T_1\) and \(T_2\) are compact from \(\Lrm^1_c(B_1)\) to \(\Lrm^1_{\rm loc}(\R^d)\), where \(\Lrm^1_c(B_1)\) is the set of \(\Lrm^1\)-functions vanishing outside \(B_1\). Indeed, by Lemma~\ref{compact} below, for every \(s>0\) the operator
\[
 ({\rm Id}-\Delta)^{-s/2}w =\F^{-1}\bigl[(1+4\pi^2|\xi|^2)^{-s/2}\hat w(\xi)\bigr]
 \]
is compact from \(\Lrm^1_c(B_1)\) to \(\Lrm^p(\R^d)\) for \(1<p<p(d,s)\) and by~\cite[Theorem 5.2.7]{Grafakos14book1}  the symbols \(m_1\) and \(m_2\) are \(\Lrm^p\)-multipliers. We conclude in particular that
\[
  \text{$\{g_j+h_j\}_j$ is precompact in  \(\Lrm^1_\loc(\R^d)\).}
\]

From~\eqref{infinitesimal} we further get
\begin{equation}\label{eq:2}
\dprb{ f_j,\varphi} = \dprb{ T_0[\chi V_j] ,\varphi} = \dprb{  \chi V_j ,T_0^*[\varphi]}\to 0  \qquad\text{for every \(\varphi \in \D(\R^d;\R^n)\),}
\end{equation}
where \(T_0^*:\S(\R^d;\R^n)\mapsto\S(\R^d,\R^m)\) is the adjoint of \(T_0\). Since \(\chi u_j\ge 0\), \eqref{pezzi5} gives that
\[
f_j^{-}\defeq\max\{0,-f_j\}\le |g_j+h_j|.
\]
As shown above, the family \(\{g_j+h_j\}_j\) is precompact in  \(\Lrm^1_\loc(\R^d)\) and thus the previous inequality implies the local equi-integrability of \(\{f_j^{-}\}\). Together with \eqref{eq:1}, \eqref{eq:2} and Lemma~\ref{fava} below this yields \(f_j \to 0\) in \(\Lrm^1_{\rm loc}(\R^d)\) and thus that the sequence \(\{\chi u_j\}\) is precompact in \(\Lrm_{\rm loc}^1(\R^d)\). Since also \(\chi u_j\weakstarto  \chi \nu\) by \eqref{conv}, we deduce that \(\chi \nu\in \Lrm^1(\R^d)\), which implies \eqref{tan1}, Moreover, 
\[
\chi u_j \to \chi \nu\qquad\textrm{in \(\Lrm^1(\R^d)\)},
\]
which, taking  into account \eqref{quasi}, implies \eqref{tan2}.  
\end{proof}

\subsection{General operators}

We now treat the general situation, namely the case of a measure \(\mu \in \M(\Omega;\R^n)\) satisfying
\begin{equation}\label{korder}
\A\mu =\sum_{|\alpha|\le k }A_\alpha \partial^\alpha \mu = 0  \qquad \textrm{in \(\D'(\Omega;\R^n)\).}
\end{equation}

\begin{proof}[Proof of Theorem~\ref{main}]
As before, let us set 
\[
E\defeq\setBB{x\in \Omega}{\frac{\d\mu}{\d|\mu|}(x)\notin \Lambda_\A}
\]
and  assume that \(|\mu|^s(E)>0\). Arguing as in the proof for first-order operators, we may find a point \(x_0\in E\) satisfying (i), (ii) above and also
\begin{itemize}
\item[(iii')] for the polar vector it holds that
\[
P_0 \defeq \frac{\d \mu}{\d|\mu|}(x_0)\notin \Lambda_\A
\]
and there is a positive constant \(c>0\) such that \(|\AA^k(\xi)P_0\big|\ge c|\xi|^k\) for \(\xi \in \R^d\).
\end{itemize}
We will show that  (i),~(ii) and~(iii') together imply  \eqref{tan1} and \eqref{tan2}, and thus yield the desired contradiction.

Assuming that \(x_0=0\), we note that~\eqref{korder} and a simple scaling argument give
\begin{equation*}\label{pezzi11}
\A^k\big(T^{r}_\sharp  \mu\big)+\sum_{h=0}^{k-1}  \A^{h}\big(r^{k-h} T^{r}_\sharp \mu\big)=0,
\end{equation*}
where \(\A^h\defeq \sum_{|\alpha|=h} A_\alpha \partial^\alpha\) is the \(h\)-homogeneous part of the operator $\A$.
Hence, with \(\nu_j\) defined as in (ii) and  \(c_j=|\mu|^s(B_{r_j})^{-1}\),
\begin{equation*}\label{pezzi12}
\begin{split}
\sum_{|\alpha|=k} A_\alpha\partial^\alpha (P_0  \nu_j)=\sum_{|\alpha|=k} A_\alpha\partial^\alpha \big(P_0 \nu_j-c_jT^{r_j}_\sharp  \mu\big)
-\sum_{h=1}^{k-1}  \A^{h}\big(r_j^{k-h} c_jT^{r_j}_\sharp \mu\big).
\end{split}
\end{equation*}
Mollification and localization now yield
\begin{equation}\label{equazione}
\sum_{|\alpha|=k} A_\alpha\partial^\alpha (P_0  \chi u_j)=\sum_{|\alpha|=k} A_\alpha \partial^\alpha (\chi V_j)+R_j.
\end{equation}
Here, as before,
\[
u_j\defeq \nu_j* \varphi_{\eps_j}, \qquad
V_j=[P_0 \nu_j-c_jT^{r_j}_\sharp  \mu]*\varphi_{\eps_j},
\]
where \(\chi\in \D(B_{3/4})\) with $0 \leq \chi \leq 1$, $\chi \equiv 1$ on $B_{1/2}$, and \(\varphi_{\eps_j}\) is a sequence of mollifier such that \eqref{quasi} is satisfied.  In particular, by (i), \(\|\chi V_j\|_{\Lrm^1}\to 0\). Moreover, the remainder term \(R_j\) can be written as a finite sum of smooth-coefficient  partial differential operators of order {\em at most \(k-1\)} applied to smooth  functions with  bounded \(\Lrm^1\)-norm and compact support:
\begin{equation*}\label{rem}
R_j=\sum_{h=0}^{k-1} \sum_{|\alpha|=h} b_{\alpha}(x) \partial^\alpha z^{\alpha}_j,
\end{equation*}
where \(b_{\alpha}(x)\in \D(B_{3/4})\), the functions \(z_j^{\alpha}\) are smooth and compactly supported,  and  \(\sup_{j} \|z_j^{\alpha}\|_{\Lrm^1}\le C\) for some constant \(C \). Namely, \(R_j=R_j^1+R_j^2+R_j^2\) where
\begin{align*}
R^1_j&=\sum_{|\alpha| =k} \sum_{\substack{\beta+\gamma=\alpha\\ |\gamma|\ge 1 }}c_{\beta\gamma} \partial^\gamma \chi\, {\partial^\beta}(A_\alpha P_0  \tilde \chi u_j),\\
R^2_j&=\sum_{|\alpha|= k} \sum_{\substack{\beta+\gamma=\alpha\\ |\gamma|\ge 1 }}c_{\beta\gamma} \partial^\gamma \chi\, {\partial^\beta}(A_\alpha \tilde \chi V_j),  
\\
R^3_j&=\sum_{|\alpha|\le k-1} \sum_{|\alpha|=h} \chi\, \partial^\alpha\Big( \tilde \chi A_{\alpha}(r_j^{k-h} c_jT^{r_j}_\sharp \mu)*\varphi_{\eps_j}\Big)
\end{align*}
with \(c_{\beta\gamma}\in \R\), and \(\tilde \chi\in \D(B_1)\) is  identically equal to $1$ on the support of \(\chi\).

By taking the Fourier transform of  \eqref{equazione} and performing the same computations as in the first part, but now multiplying with $[\AA^k(\xi)P_0]^*$ instead of $[\AA(\xi)P_0]^*$, we obtain 
\begin{equation}\label{eqauzione2}
\chi u_j=S_0 [\chi V_j] +S_1[\chi u_j]+\widetilde R_j,
\end{equation}
where \(S_0\)   and  \(S_1\) are given by 
\begin{align*}
S_0[V]&=\F^{-1} \Bigg[\frac{[\AA^k(\xi)P_0]^* \AA^k(\xi)\, \hat {V}(\xi) }{1+|\AA^k(\xi)P_0|^2} \Bigg],\\
S_1[u]&=\F^{-1} \Bigg[\frac{(1+4\pi^2|\xi|^2)^k}{1+|\AA^k(\xi)P_0|^2} \cdot \frac{\hat {u}(\xi)}{(1+4\pi^2|\xi|^2)^{k}}\Bigg].
\end{align*}
Applying the  H\"ormander--Mihlin multiplier theorem and arguing as for first-order operators, we deduce that 
\[
\sup_{\lambda \ge 0} \, \lambda\, \L^d \bigl(\{ |S_0 [\chi V_j] |>\lambda\}\bigr)\le C \|\chi V_j\|_{\Lrm^1(B_1)}\to 0 \qquad\textrm{and}\qquad S_0 [\chi V_j]\weakstarto 0.
\]
Moreover, the family \(\{S_1[\chi u_j]\}\) is precompact in \(\Lrm_\loc^1(\R^d)\).  To conclude the proof it is enough to show that \(\{\widetilde R_j\}\) is precompact in  \(\Lrm^1_\loc(\R^d)\), since then the application of Lemma~\ref{fava} as in the first part  will imply the validity of \eqref{tan1} and \eqref{tan2}. The generic term of \(\widetilde R_j \) can be written as 
\begin{equation*}\label{fj}
f_j^\alpha= Q\circ({\rm Id}-\Delta)^{-\frac{k}{2}} \circ P_{\alpha} \circ ({\rm Id}-\Delta)^{\frac{|\alpha|-k}{2}} [z^\alpha_j] 
\end{equation*}
where \(0\le |\alpha|\le (k-1)\), \(\sup_j\|z^\alpha_j\|_{\Lrm^1}\le C\),
\[
Q[z]=\F^{-1} \Bigl[(1+|\AA^k(\xi)P_0 |^2)^{-1}(1+4\pi^2|\xi|^2)^{k/2}\, \AA^k(\xi) \hat {z}(\xi)\Bigr] ,
\]
and \(P_\alpha\) is the $k$'th-order pseudo-differential operator given by
\[
P_{\alpha}[z](x)=\int b_{\alpha}(x)\frac{(2\pi\ii)^{|\alpha|}\xi^\alpha}{(1+4\pi^2|\xi|^2)^{\frac{|\alpha|-k}{2}}} \, \hat{z}(\xi) \, \ee^{2\pi\ii x \cdot \xi} \dd \xi,  \qquad x \in \R^d.
\]
The composition \(({\rm Id}-\Delta)^{-k/2} \circ P_{\alpha}\) is a pseudo-differential operator of order \(0\), see~\cite[Theorem 2, Chapter VI]{Stein93book},  and thus bounded from  \(\Lrm^p(\R^d)\) to itself, see~\cite[Proposition 4, Chapter VI]{Stein93book}. By (iii') and the   H\"ormander--Mihlin multiplier theorem,  also \(Q\) is a bounded  operator  from  \(\Lrm^p(\R^d)\) to itself. Since \(|\alpha|\le k-1\), Lemma~\ref{compact} below implies that \(({\rm Id}-\Delta)^{(|\alpha|-k)/2}\) is compact from \(\Lrm^1_c(B_1)\) to \(\Lrm^p(\R^d)\) for \(1<p<p(d,|\alpha|-k)\). We conclude that \(\{f_j^\alpha\}\) is precompact in \(\Lrm^1_\loc(\R^d)\). The validity of \eqref{tan1} and \eqref{tan2} now follows from \eqref{eqauzione2} by arguing as before.
\end{proof}

\subsection{Auxiliary results}

Finally, we prove the two simple technical lemmas that have been used in the proofs above. The first is an $\Lrm^1$-compactness result in the spirit of the Sobolev embedding theorems. Since we have not been able to find a reference we provide its simple proof.

\begin{lemma}\label{compact}
For  \(u\in \S(\R^d)\) and \(s>0\) define
\[
({\rm Id}-\Delta)^{-s/2}u:=\F^{-1}\bigl[(1+4\pi^2|\xi|^2)^{-s/2}\hat u (\xi)\bigr].
\] 
Then, \(({\rm Id}-\Delta)^{-s/2}\) extends to a compact map from \(\Lrm^1_c(B_1)\) to \(\Lrm^p(\R^d)\) for \(1\le p<p(d,s)\), where
\[
p(d,s)\defeq
\begin{cases}
\dfrac{d}{d-s}\qquad&\textrm{if \(s<d\),}\\
\infty &\textrm{otherwise,}
\end{cases}
\]
and \(\Lrm^1_c(B_1)\subset \Lrm^1(\R^d)\) is the set of \(\Lrm^1\)-functions supported in \(B_1\).
\end{lemma}

\begin{proof}
For \(u\) in the Schwartz class we can write
\[
({\rm Id}-\Delta)^{-s/2}u=b(s,d)* u
\]
where \(b(s,d)=\F^{-1}[(1+4\pi|\xi|^2)^{-s/2}]\) is the Bessel potential of order \(s\), see~\cite[Section 6.1.2]{Grafakos14book2}. By classical estimates~\cite[Proposition 6.1.5]{Grafakos14book2},  \(b(s,d)\in \Lrm^p\) for \(1\le p<p(d,s)\) so that by Young's inequality for convolutions, \(({\rm Id}-\Delta)^{-s/2}u\in \Lrm^p\) for \(1\le p<p(d,s)\) (actually also for \(p=p(d,s)\) if \(s\ne d\)). For every  \(\eps>0\) we can write
\[
  b(s,d)=b_{1,\eps}+b_{2,\eps}  \qquad\text{with}\qquad
  b_{1,\eps}\in \Crm^1_c(\R^d) \quad\text{and}\quad
  \|b_{2,\eps}\|_{\Lrm^1}<\eps,
\]
see~\cite[Proposition~6.1.6]{Grafakos14book2}. Thus,
\[
  ({\rm Id}-\Delta)^{-s/2}u=b_{1,\eps}* u+b_{2,\eps}* u\defeq T_{1,\eps}[u]+T_{2,\eps}[u].
\]
Because \(b_{1,\eps}\in \Crm^1_c(\R^d)\),  \(T_{1,\eps}\) is  compact from \(\Lrm^1_c(B_1)\) to \(\Lrm^1(\R^d)\).  Moreover,
\[
\| ({\rm Id}-\Delta)^{-s/2}-T_{1,\eps}\|_{\Lrm^1\to\Lrm^1}\le \|T_{2,\eps}\|_{\Lrm^1\to\Lrm^1}\le \eps,
\]
so that  \( ({\rm Id}-\Delta)^{-s/2}\) is the limit in the uniform topology of compact operators and thus compact as well. The conclusion of the lemma now follows by H\"older's inequality.
\end{proof}

The second lemma is an easy consequence of the Vitali convergence theorem:

\begin{lemma}\label{fava}
Let \(\{f_j\} \subset \Lrm_c^1(B_1)\) be a family of functions such that
\begin{itemize}
\item[(a)] \(f_j\weakstarto 0\) in \( \D'(B_1)\);
\item[(b)] The negative parts of \(f_j\) tends to zero in measure in, i.e.\
 \[
\lim_{j\to \infty}\, \L^d\big(\{f_j^->\lambda\}\big)=0  \qquad\text{for every \(\lambda>0\);}
\]
\item[(c)]  the sequence of negative parts \(\{f_j^{-}\}\) is equi-integrable,
\[
\qquad \lim_{\L^d(E)\to 0} \sup_{j\in \N} \, \int_{ B_1} f_j^- \dd x = 0.
\]
\end{itemize}
Then, \(f_j \to0\) in  \(\Lrm^1_\loc(B_1)\).
\end{lemma}

\begin{proof}
Let \(\varphi\in \D(B_1)\), \(0\le \varphi\le 1\). It is enough to show that 
\begin{equation}\label{tesi}
\lim_{j\to \infty}\int \varphi |f_j|  \dd x =0.
\end{equation}
We write 
\[
\int \varphi |f_j| \dd x =\int \varphi f_j \dd x +2\int \varphi f_j^- \dd x
\le \int \varphi f_j \dd x +2\int f_j^- \dd x.
\]
The first term on the right-hand side goes to \(0\) as \(j\to \infty\) by assumption~(a). Thanks to the Vitali convergence theorem, assumptions (b) and (c) further give that also the third term vanishes in the limit. Hence, \eqref{tesi} follows.
\end{proof}

\section{Applications}\label{appl}
Theorems~\ref{corBV},~\ref{corHigher} and~\ref{corBD} follow from Theorem~\ref{main} simply by applying it  to the differential constraints that gradients, higher gradients, or symmetrized gradients, respectively, have to satisfy.

\begin{proof}[Proof of Theorem~\ref{corBV}]
Let  \(\mu=(\mu^k_{j})\in \M(\Omega;\R^{\ell \times d})\) be the (distributional) gradient of a function \(u\in \BV(\Omega;\R^\ell)\), \(\mu=Du\). Then, 
\[
0=\partial_i \mu^{k}_{j}-\partial_j  \mu^{k}_{i}\qquad i,j=1,\ldots, d;\; k=1,\ldots,\ell.
\]
Setting
\[
  \A \mu := \big( \partial_j \mu^k_i - \partial_i \mu^k_j \bigr)_{i,j=1,\ldots,d;\, k=1,\ldots,\ell} \;,
\]
it is a simple algebraic exercise, carried out for instance in~\cite[Remark~3.5(iii)]{FonsecaMuller99}, to compute that
\[
\Lambda_\A=\setb{ a\otimes \xi \in \R^{\ell \times d} }{ a\in \R^\ell,\, \xi \in \R^d \setminus \{0\} }.
\]
Corollary~\ref{corBV} then follows directly from Theorem~\ref{main}.
\end{proof}

\begin{proof}[Proof of Theorem~\ref{corHigher}]
For the operator
\[
  \A \mu := \Bigl( \partial_j \mu^k_{\alpha_1 \cdots \alpha_h i \alpha_{h+2} \cdots \alpha_r} - \partial_i \mu^k_{\alpha_1 \cdots \alpha_h j \alpha_{h+2} \cdots \alpha_r} \Bigr)_{i,j,\alpha_1,\ldots,\alpha_n=1,\ldots,d;\, k=1,\ldots,\ell;\, h=1,\ldots,r}
\]
one can see that $\A \mu = 0$ if and only if $\mu$ is an $r$'th-order derivative, and furthermore compute that
\[
\Lambda_\A=\setb{ a\otimes \xi \otimes \cdots \otimes \xi \in \mathrm{SLin}^r(\R^d;\R^\ell) }{ a\in \R^\ell,\, \xi \in \R^d \setminus \{0\} };
\]
see~\cite[Example~3.10(d)]{FonsecaMuller99} for the details.
\end{proof}

\begin{proof}[Proof of Theorem~\ref{corBD}]
Let  \(\mu=(\mu_{j}^k)\in \M(\Omega ,\R^{d \times d}_{\mathrm{sym}} )\) be  the  (distributional) symmetrized gradient  of  \(u\in \BD(\Omega)\), \(\mu=Eu \). Then, by direct computation, see~\cite[Example 3.10(e)]{FonsecaMuller99},
\[
0=\A \mu:= \biggl( \sum_{i=1}^d \partial_{ik} \mu_{i}^j+\partial_{ij} \mu_{i}^k-\partial_{jk} \mu_{i}^i-\partial_{ii} \mu_{j}^k \biggr)_{j,k=1,\ldots,d} \;.
\]
These equations are often called the \emph{Saint-Venant compatibility conditions} in applications. Hence, for \(M\in  \R^{d \times d}_{\mathrm{sym}}\),
\[
-(4\pi)^{-2}\AA(\xi)M=(M\xi)\otimes \xi+\xi\otimes (M\xi)-({\rm tr} M) \, \xi\otimes \xi -|\xi|^2 M,
\]
which gives
\[
\Ker \AA^2(\xi) = \Ker \AA(\xi) = \setb{a\otimes \xi+\xi \otimes a}{a\in \R^d,\, \xi \in \R^d}.
\]
 Theorem~\ref{main} now implies the conclusion.
\end{proof}

\begin{proof}[Proof of Corollary~\ref{corU}]
The only fact to show in addition to the assertion of Corollary~\ref{corBD} is that $a(x) \cdot b(x) = 0$. For $Eu$ we have the Lebesgue--Radon--Nikod\'{y}m decomposition $Eu = \Ecal u \, \L^d + E^s u$ and thus
\[
  \diverg u = \mathrm{tr} (\Ecal u) \, \L^d + a(x) \cdot b(x) \, \abs{E^s u}
\]
Since $\diverg u \in \Lrm^2(\Omega)$, we must have $a(x) \cdot b(x) = 0$ for $\abs{E^s u}$-almost every $x \in \Omega$.
\end{proof}

Before proving  Theorem~\ref{thm:curr}, let us recall  some simple facts concerning (Euclidean) currents and multi-linear algebra. We  refer to~\cite{Federer69} for more details.

Given a finite dimensional vector space \(V\)   we let \(\Lambda_k(V)\) be  the set of \(k\)-vectors and  \(\Lambda^{k}(V)\simeq (\Lambda_{k}(V))^*\) be the set of \(k\)-covectors. If \(v\in \Lambda_k(V)\) and  \(\eta\in \Lambda^1(V)\), then the interior product of \(\eta\) with \(v\) is the \((k-1)\)-vector \(v\imult\eta \in \Lambda_{k-1}(V)\) defined by duality as  \(\langle v\imult\eta,\omega\rangle\defeq \langle v, \eta\wedge \omega\rangle\) for every \(\omega\in \Lambda^{k-1}(V)\), see ~\cite[Section 1.5]{Federer69}. 

Following~\cite[Section 4.1.7]{Federer69}, we let 
 \[
 \D^k(\Omega):=\D(\Omega,\Lambda^k(\R^d))\qquad  \textrm{and}\qquad \D_k(\Omega):=\D'(\Omega,\Lambda_k(\R^d))
 \]
be the sets of compactly supported  \(k\)-differential forms with smooth coefficients and  the set of \(k\)-dimensional currents, respectively. For \(T\in \D_k(\Omega)\) the boundary \(\partial T\in \D_{k-1}(\Omega)\) is defined by duality with the exterior differential via  \(\langle \partial T, \omega\rangle\defeq \langle T,d \omega\rangle\), where \(\omega\in \D^{k-1}(\Omega)\). One easily checks that 
\begin{equation}\label{boundary}
\partial T=-\sum_{i=1}^d \partial_i T\imult dx^i,
\end{equation}
see~\cite[p.~356]{Federer69}. Here,  for  \(T\in \D_k(\Omega)\) and  \(\eta\in \Crm^\infty(\Omega;\Lambda^1(\R^d))\), \(T\imult \eta\in \D_{k-1}(\Omega)\) is  defined as \(\langle T\imult \eta, \omega\rangle\defeq \langle T,\eta\wedge \omega\rangle\), \(\omega\in \D^{k-1}(\Omega)\) and $\partial_i T \in \Dcal_k(\Omega)$ is defined by duality via $\langle \partial_i T, \phi \, dx^{j_1} \wedge \cdots \wedge dx^{j_k} \rangle = - \langle T, \partial_i \phi \, dx^{j_1} \wedge \cdots \wedge dx^{j_k} \rangle$.

We endow \(\Lambda_k(\R^d)\) with the mass norm, see~\cite[Section 1.8]{Federer69}. A \(k\)-current is said to have finite mass if it can be extended to a \(\Lambda_k(\R^d)\)-valued (finite) Radon measure and we  let \(\|T\|\) be the total variation of \(T\) and
\[
  \vec T:=\frac {\dd T}{\dd \|T\|},
\]
see~\cite[Section 4.1.7]{Federer69}. In this context, the Radon--Nikod\'{y}m theorem reads as \(T=\vec T\|T\|\). We denote by \(\mathbf N_k(\Omega)\) the set of \(k\)-dimensional normal  currents, i.e.\ those currents such that \(T\) and \(\partial T\) both have finite mass.  Note that  the boundary of a \(k\)-dimensional normal current \(T\) can be seen as a \(\Lambda_{k-1}(\R^d\))-valued Radon measure, \(\partial T\in \M(\Omega;\Lambda_{k-1}(\R^d))\). 

\begin{proof}[Proof of Theorem~\ref{thm:curr}]
Let us set
 \[
 \boldsymbol{T}\defeq(T_1,\ldots,T_r)\in \M(\Omega;\Lambda_{k_1}(\R^d)\times\ldots\times \Lambda_{k_r}(\R^d))
 \]
and note that the assumption of  Theorem~\ref{thm:curr} can be rewritten as 
\[
\A\boldsymbol{T}\defeq (\partial T_1,\ldots,\partial T_r)\in \M(\Omega;\Lambda_{k_1-1}(\R^d)\times\ldots\times \Lambda_{k_r-1}(\R^d)).
\]
By applying Theorem~\ref{main} in conjunction with  Remark~\ref{rmk:rhs} we deduce that  for \(|\boldsymbol T|^s\)-almost every \(x\in\Omega\) there exists \(\xi_x\ne 0\) such that 
\begin{equation}\label{fame}
\frac{\dd \boldsymbol{T}}{\dd |\boldsymbol{T}|}(x)\in  \Ker \AA(\xi_x)\,.
\end{equation}
Thanks to  \eqref{boundary},  one easily checks that  for \(\boldsymbol v=(v_1,\ldots,v_r) \in \Lambda_{k_1}(\R^d)\times\ldots\times \Lambda_{k_r}(\R^d)\) it holds that
\begin{equation}\label{ker}
\AA(\xi)\boldsymbol v=-2\pi \ii\, \big(v_1\imult \omega_\xi,\ldots, v_r\imult\omega_\xi\big)\in \Lambda_{k_1-1}(\R^d)\times\ldots\times \Lambda_{k_r-1}(\R^d),
\end{equation}
where \(\omega_\xi\in \Lambda^1(\R^d)\) is  defined as \(\omega_\xi(w)\defeq w\cdot \xi\), \(w\in \R^d\).

Let \(\mu\in \M_+(\Omega)\) be  as in the statement of the theorem and  note that, since \(\mu\aac\|T_i\|\) for every \(i=1,\ldots,r\), the  Radon-Nikod\'{y}m  derivatives \(\frac{\dd |\boldsymbol{T}|}{\dd \|T_i\|}\) and  \(\frac{\dd T_i}{\dd |\boldsymbol{T}|}\) exist \(\mu\)-almost everywhere. Then,
\begin{equation}\label{noia}
 \vec T_i=\frac{\dd |\boldsymbol{T}|}{\dd \|T_i\|} \, \frac{\dd T_i}{\dd |\boldsymbol{T}|}\,.
\end{equation}
Since clearly \(\mu^s\aac |\boldsymbol T|^s\), the first part of the conclusion  with \(\omega_x=\omega_{\xi_x}\) follows from \eqref{fame}, \eqref{ker} and \eqref{noia}. 
It is now a simple exercise in linear algebra to see that the second part of the statement is equivalent to the first one.
\end{proof}

\begin{proof}[Proof of Corollary~\ref{cor:1curr}]
By Theorem~\ref{thm:curr}, assumption~(i) implies that for \(\mu^s\)-almost every \(x \in \R^d\) there exists a \((d-1)\)-dimensional subspace \(V_x\) such that
\[
  \vec T_1(x),\ldots,\vec T_d(x)\in V_x.
\]
Assumption~(ii) hence gives that \(\mu^s=0\), which is the desired conclusion.
\end{proof}

\begin{proof}[Proof of Corollary~\ref{corDiv}]
Let \(\mu = (\mu^k_j) \in \M(\Omega;\R^{d\times d})\) and let 
\[
\A \mu:= \diverg \mu = \biggl(\sum_{j=1}^{d}\partial_j \mu^k_j \biggr)_{k=1,\ldots,d} \; .
\]
Then, for \(M\in \R^{d\times d}\), \(\AA(\xi)M=(2\pi \ii)M\xi\),
so that
\[
  \Lambda_\A = \setb{ M \in \R^{d \times d} }{ \mathrm{rank}\, M \leq d-1 }.
\]
The conclusion  follows from Theorem~\ref{main} and Remark~\ref{rmk:rhs}.
\end{proof}

We will now  show how to obtain Theorems~\ref{cor:rad} and~\ref{cor:curr} from Corollary~\ref{cor:1curr}. In order to do so, we  assume the reader to be familiar with the work of Alberti \& Marchese~\cite{AlbertiMarchese15} concerning differentiability of Lipschitz functions, with the definition of metric currents given in~\cite{AmbrosioKirchheim00}, as well as with the work of Schioppa in~\cite{Schioppa15}. We refer to these papers also for notations and definitions.

Let us start with  the following lemma, which is essentially~\cite[Corollary 6.5]{AlbertiMarchese15}. 
\begin{lemma}\label{equiv}
Let \(\mu\in \M_+(\R^d)\) be a finite  positive Radon measure. Then the following are equivalent:
\begin{itemize}
\item[(i)] The decomposability bundle of \(\mu\) (see~\cite[Section 2.6]{AlbertiMarchese15}) is of full dimension,  \(V(\mu,x)=\R^d\) for \(\mu\)-a.e.\ \(x \in \R^d\).
\item[(ii)] There are \(d\) normal one-dimensional currents \(T_1=\vec T_1 \|T_1\|,\ldots, T_d = \vec T_d \|T_d\|\in \mathbf{N}_1(\R^d)\) 
 such that  \(\mu\aac \|T_i\|\) for \(i=1,\ldots,d\), and 
 \[
 \Span\bigl\{\vec T_1(x),\ldots,\vec T_d(x)\bigr\}=\R^d\qquad\textrm{for \(\mu\)-a.e.\ \(x \in \R^d\).}
\]
\end{itemize}
\end{lemma}

\begin{proof}
The implication (i) $\Rightarrow$ (ii)  is obtained by choosing (in a measurable way) for \(\mu\)-a.e.\ \(x \in \R^d\) a  basis \(\{e_1(x),\ldots,e_d(x)\}\) of \(V(\mu,x)\) and by applying to each \(e_i\) the implication  (i) $\Rightarrow $ (ii) of~\cite[Corollary 6.5]{AlbertiMarchese15}. For  the other  implication,  simply notice that,  by  the implication (ii) $\Rightarrow$ (i) of~\cite[Corollary 6.5]{AlbertiMarchese15},  \(\vec T_i(x)\in V(\mu,x)\) for \(\mu\)-a.e.\ \(x \in \R^d\).
\end{proof}

\begin{proof}[Proof of Theorem~\ref{cor:rad}] By~\cite[Theorem 1.1]{AlbertiMarchese15} the assumptions in the statement of the theorem are equivalent to \(V(\mu,x)=\R^d\) for \(\mu\)-a.e.\ \(x \in \R^d\). This implies that \(\mu\aac \L^d\) by Lemma~\ref{equiv} and Corollary~\ref{cor:1curr}. 
\end{proof}

\begin{proof}[Proof of Theorem~\ref{cor:curr}] 
By~\cite[Theorem 1.3]{Schioppa15} the mass measure \(\|T\|\) associated with a \(d\)-dimensional metric current \(T\in \mathbf M^{\rm met}_{d}(\R^d)\)  admits  \(d\) independent Alberti representations, which, by the very definition of decomposability bundle, see~\cite[Section 2.6]{AlbertiMarchese15}, implies that \(V(\|T\|, x)=\R^d\) for \(\|T\|\)-a.e.\ \(x \in \R^d\). Theorem~\ref{cor:curr} hence  follows from Theorem~\ref{cor:rad}, see also the discussion after Theorem~1.3 in~\cite{Schioppa15}. 
\end{proof}



\providecommand{\bysame}{\leavevmode\hbox to3em{\hrulefill}\thinspace}
\providecommand{\MR}{\relax\ifhmode\unskip\space\fi MR }
\providecommand{\MRhref}[2]{%
  \href{http://www.ams.org/mathscinet-getitem?mr=#1}{#2}
}
\providecommand{\href}[2]{#2}

\end{document}